\documentclass{article}%
\usepackage{amsmath}
\usepackage{amsfonts}
\usepackage{amsthm}
\usepackage{amssymb}
\usepackage{graphicx}%
 \newtheorem{thm}{Theorem}[section]
 \newtheorem{cor}[thm]{Corollary}
 \newtheorem{lem}[thm]{Lemma}
 
 \theoremstyle{definition}
 \newtheorem{defn}[thm]{Definition}
 \theoremstyle{remark}
 \newtheorem{rem}[thm]{Remark}
 \newtheorem*{ex}{Example}
 \numberwithin{equation}{section}
\usepackage{tikz,xypic} 
\usetikzlibrary{calc}
\usetikzlibrary{positioning}
\tikzset{
  treenode/.style = {align=center, inner sep=0pt, text centered,
    font=\sffamily},
 arn_n/.style = {treenode, circle, white, font=\sffamily\bfseries, draw=black,
    fill=white, text width=0.6ex},
  arn_p/.style = {treenode, circle, white, font=\sffamily\bfseries, draw=black,
    fill=white, text width=0.6ex},
}

\DeclareMathOperator{\Aut}{Aut}
\newcommand{\IP}{\ensuremath{I(P,\mathbb{F})}}
\newcommand{\A}{\ensuremath{\mathcal{A}}}
\newcommand{\B}{\ensuremath{\mathcal{B}}}
\newcommand{\FX}{\ensuremath{\mathbb{F}\langle X \rangle}}

\makeindex
\begin{document}

\title{Gradings on Incidence Algebras and their Graded Polynomial Identities}
\author{Humberto Luiz Talpo\\{\small Departamento de Matem\'{a}tica - Universidade Federal de S\~ao Carlos} \\{\small Rod. Washington Lu\'is, Km 235 - C.P. 676, 13565-905 S\~ao Carlos-SP, Brasil.}\\{\small \texttt{htalpo@ufscar.br}}
\and Waldeck Sch\"utzer\\{\small Departamento de Matem\'{a}tica - Universidade Federal de S\~ao Carlos}\\{\small Rod. Washington Lu\'is, Km 235 - C.P. 676, 13565-905 S\~ao Carlos-SP, Brasil.}\\{\small \texttt{waldeck@dm.ufscar.br}}}
\date{}
\maketitle

\begin{abstract}
Let $P$ a locally finite partially ordered set, $\mathbb{F}$ a field, $G$ a group, and $\IP$ the incidence algebra of $P$ over $\mathbb{F}$. We describe all the inequivalent elementary $G$-gradings on this algebra.
If $P$ is bounded, $\mathbb{F}$ is a infinite field of characteristic zero, and $\A,\B$ are both elementary $G$-graded incidence algebras satisfying the same $G$-graded polynomial identities, and the automorphisms group of $P$ acts transitively on the maximal chains of $P$, we show that $\A$ and $\B$ are graded isomorphic. 
\end{abstract}

\section{Introduction}
\quad Given a fixed field $\mathbb{F}$, Stanley \cite{Stanley} showed that if $P$ and $Q$ are both locally finite partially ordered sets whose incidence algebras $\IP$ and $I(Q,\mathbb{F})$ are isomorphic, then $P$ and $Q$ are isomorphic. Feinberg \cite{Feinberg} showed that if $P$ is a locally finite partially ordered set whose chains have length no greater than $n-1$, then the incidence algebra $\IP$ over an infinite field satisfies the standard polynomial identity of degree $2n$, the equivalent of the Amitsur-Levitzki theorem. Berele \cite{Berele} showed that the polynomial identities satisfied by $\IP$ under the conditions that $P$ is a bounded poset with bound $n$ and $\mathbb{F}$ is an infinite field, are precisely the polynomial identities satisfied by $I(C_n,\mathbb{F})$, where $C_n\subset P$ is a chain of length $n$.

The incidence algebra $\IP$ of a finite poset $P$ with $n$ elements over a field $\mathbb{F}$ has a natural embedding in $M_n(\mathbb{F})$ the algebra of $n\times n$ matrices over $\mathbb{F}$.
If $G$ is an abelian group, then, for the algebra $M_n(\mathbb{F})$, there are two important classes of $G$-gradings: the elementary gradings and the fine gradings. In \cite{Bahturin} it was proved that if $\mathbb{F}$ is an algebraically closed field, every $G$-grading on $M_n(\mathbb{F})$ is a tensor product of an elementary and a fine grading. For the algebra $UT_n(\mathbb{F})$, the algebra of $n\times n$ upper triangular matrices over $\mathbb{F}$, it was proved in \cite{Valenti} that if $\mathbb{F}$ is an algebraically closed field of characteristic zero then every finite $G$-grading is isomorphic to elementary one. Furthermore isomorphic gradings satisfy the same graded identities. In \cite{Plamen} the authors studied the elementary gradings on the algebra $UT_n(\mathbb{F})$ over an infinite field. They described these elementary gradings by means of the graded identities that they satisfy. Namely, they proved that there exist $|G|^ {n-1}$ non isomorphic elementary gradings on $UT_n(\mathbb{F})$ by a finite group $G$, and showed that non isomorphic gradings give rise to different graded identities.

In this work, we extend the notion of elementary grading to incidence algebras of arbitrary locally finite partially ordered sets,
and further describe all the inequivalent ones. We provide a grading
counting formula when $P$ and $G$ are finite. When $P$ is bounded and $\mathbb{F}$ is a infinite field of characteristic zero, we show that whenever $\A$ and $\B$ are both $G$-graded incidence algebras of $P$ over $\mathbb{F}$ satisfying the same $G$-graded polynomial identities and the automorphisms group of $P$ acts transitively on the maximal chains of $P$, then $\A$ and $\B$ are isomorphic as $G$-graded algebras.

\section{Preliminaries}
\quad A set $P$ with a binary relation $\preceq$ is a partially ordered set
(abbreviated \emph{poset}) if $\preceq$ is reflexive, transitive and
antisymmetric. A subset $Q$ of $P$ is a \emph{subposet} when, for $x,y\in Q$, it satisfies $x\preceq y$ in $Q$ if, and only if, $x\preceq y$ in $P$.
If $P$ and $Q$ are posets, a function $\varphi:P\to Q$ is \emph{order-preserving}
if $x\preceq y$ in $P$ implies $\varphi(x)\preceq\varphi(y)$. If $\varphi$ is a bijection whose inverse is also order-preserving, we say that $\varphi$ is an \emph{isomorphism} of posets. An \emph{automorphism} of $P$ is an isomorphism
of $P$ onto itself (with respect to the same relation).

A \emph{chain} is a poset $C$ such that for any $x,y\in C$, either $x\preceq y$ or $y\preceq x$. A finite chain with $n$ distinct elements is said to have \emph{length} $n$, and an infinite chain has infinite length. All finite chains of same length $n$ are isomorphic to the chain $C_n=\{1,2\ldots,n\}$ with the usual order. A chain $C$ in $P$ is \emph{maximal} if, given any chain $C'$ in $P$, $C\subset C'$ implies either $C=C'$ or $C=P$. Every chain in $P$ can be prolonged to a maximal chain. In particular, every element of $P$ belongs to a maximal chain.

A poset $P$ is \emph{bounded} if there is an integer $n$ such that every chain in $P$ has length no greater than $n$. Otherwise $P$ is \emph{unbounded}. It is worth noting that there do exist unbounded posets all whose chains are finite.

Given $x$ and $z$ in a poset $P$, the \emph{segment} from $x$ to $z$, denoted by $[x,z]$, is the subposet $[x,z]=\{y\in P|x\preceq y\preceq z\}$.
The segment has length $n$ it has a chain of length $n$ and any other chain in it has length no greater than $n$. A poset $P$ is \emph{locally finite} if every segment of $P$ is finite.

Two elements $x, y$ in a poset $P$ are \emph{connected} if, for some positive integer $n$, there exist elements $x=x_0,x_1,x_2,\dots,x_n=y$ in $P$ with either $x_{i-1}\preceq x_{i}$ or $x_{i}\preceq x_{i-1}$ for $i=1,\dots,n$. Connectedness is an equivalence relation in $P$ whose equivalence classes are called \emph{connected components}.

Let $P$ be a locally finite poset and $\mathbb{F}$ a field. The \textbf{incidence algebra} $\IP$ of $P$ over $\mathbb{F}$ is defined as the $\mathbb{F}$-algebra of all functions \\$f:P\times P\rightarrow \mathbb{F}$ such that $f(x,y)=0$ whenever $x\npreceq y$, under the operations:
\begin{eqnarray*}
(f_1+f_2)(x,y) & = & f_1(x,y)+f_2(x,y),\\
(kf)(x,y) & = & kf(x,y),\\
(f_1f_2)(x,y) & = & \sum_{x\preceq z\preceq y}f_1(x,z)f_2(z,y).
\end{eqnarray*}
for all $f_1,f_2\in I(P,\mathbb{F})$, $x,y,z\in P$ and $k\in \mathbb{F}$.
Since $P$ is locally finite, it is clear that the above summation is finite.
Whenever $Q$ is a subposet of $P$, $I(Q,\mathbb{F})$ is a subalgebra
of $\IP$ coinciding with the set of functions $f$ such that $f(x,y)=0$ whenever $x\notin Q$ or $y\notin Q$.

From now on, if the poset $P$ and the field $\mathbb{F}$ are fixed, to simplify notation we shall denote $\IP$ simply by $\A_P$. If there is
no risk of confusion, we shall drop the subscript $P$.

For each $x,y\in P$, $\A$ contains the function $e_{xy}$ such that
\[
e_{xy}(u,v)=\begin{cases}
1, & \mbox{if \ensuremath{x=u\preceq v=y} }\\
0, & \mbox{otherwise.}
\end{cases}
\]
It is straightforward to check that such functions satisfy
$e_{xy}f e_{uv}=f(y,v)e_{xv}$, for all $x,y,u,v\in P$ and $f\in\A$. In
particular, $e_{xx}fe_{yy}=f(x,y)e_{xy}$ and $e_{xy}e_{uv}=e_{xv}$ if $y=u$ and 0 otherwise, hence the $e_{xx}$ are idempotents. Moreover,
every function $f\in\A$ has a unique representation as an infinite sum of the form (cf. \cite{Doubillet-Rota-Stanley}):
\[
f = \sum_{x,y\in P}f(x,y)e_{xy}.
\]
The function $\delta=\sum_{x\in P}e_{xx}$ satisfies $\delta f = f\delta = f$ for all $f\in\A$ hence it is the unit in this algebra.

The group of all $\mathbb{F}$-automorphisms of $\A$ is denoted $\Aut(\A)$. If $r\in \A$ is invertible, then $r$ determines an automorphism $\psi_r \in \Aut(\A)$, given by $\psi_r(f)=rfr^{-1}$, for each $f\in \A$. Such an automorphism is called an \emph{inner automorphism}. 

Notice that $\psi_{r^{-1}}=(\psi_r)^{-1}$, and if $r_1$ and $r_2$ are invertible in $\A$, then $\psi_{r_1}\circ \psi_{r_2}=\psi_{r_1r_2}$.
Therefore, 
\[
\mathit{Inn}(\A) = \{\psi_r|r \text{ is invertible in } \A\}
\]
is a subgroup of $\Aut(\A)$.

A function $s \in \A$ is \emph{multiplicative} if whenever $x\preceq z\preceq y$ are elements of $P$, then
\[
s(x,y)=s(x,z)s(z,y)\in \mathbb{F}^{\times},
\]
where $\mathbb{F}^{\times}$ denote the multiplicative subgroups of $\mathbb{F}$.
Notice that if $s$ is multiplicative, then $s(x, x) = 1$ for each $x\in P$. In particular, $s$ is invertible.

Recall that the Hadamard product of $f_1,f_2\in \A$ is denoted by $f_1\ast f_2$ and is given by $(f_1 \ast f_2)(x, y) = f_1(x, y)f_2(x, y)$ for each pair of $x, y \in P$. If $s$ is multiplicative, we define
\[
\begin{array}{ccccc}
M_s & : & \A & \rightarrow & \A \\
 &  & f & \mapsto & s \ast f
\end{array}
\]
for each $f\in \A$. Note that $M_s$ is an automorphism of $\A$. Such an automorphism is called \emph{multiplicative}.
\begin{rem}\label{rem:auto:mult}
Denote by $\mathit{Mult}(A) = \{M_s|s \text{ is multiplicative }\}$ the set of all multiplicative automorphisms of $\A$. It is easy to show that this is a subgroup of $\Aut (\A)$.
\end{rem}
If $\sigma$ is an automorphism of $P$, then $\sigma$ induces an automorphism $\hat\sigma$ of
$\A$ given by $(\hat\sigma(f))(x, y) = f(\sigma^{-1}(x),\sigma^{-1}(y))$, for each $f\in \A$ and each
pair of elements $x, y \in P$. So $\Aut(P)$ can be seen, via the automorphism induced $\hat\sigma$, as a subgroup of $\Aut(\A)$.

In \cite[Theorem 5]{Baclawski}, was proved the following:
\begin{thm}\label{thm:decomp:automorphism}
Let $P$ be a locally finite poset. If $\varphi$ is an automorphism of $\A$, then $\varphi =\psi_r\circ M_s\circ \hat\sigma$ for some inner automorphism $\psi_r$, some multiplicative automorphism $M_s$ and unique automorphism $\sigma$ of $P$.
\end{thm}
\section{Elementary Gradings on Incidence Algebras}
\qquad Let $\mathbb{F}$ be a field, $R$ an associative $\mathbb{F}$-algebra,
and $G$ a (multiplicative) group. A $G$-grading on $R$ is a decomposition of $R$ as a direct sum of $\mathbb{F}$-vector subspaces $R=\oplus_{g\in G}R_g$ such that $R_gR_h\subseteq R_{gh}$ for every $g,h\in G$. The subspace $R_g$ is the $g$-th homogeneous component in the grading, and its elements are said to be homogeneous of degree $\deg r=g$. $R_1$ is the neutral (or identity) component of the grading and it always contains the identity $1$ of $R$.

The incidence algebra $\A=\IP$ of a finite poset $P$ with $n$ elements over a field $\mathbb{F}$ has a natural embedding in the algebra $M_n(\mathbb{F})$ of $n\times n$ matrices over $\mathbb{F}$. Furthermore, the elements
in $P$ can be labelled by the naturals in such a way that $x_i\preceq x_j$ in $P$ implies $i\leq j$, hence $\A$ can be identified with the subalgebra $UT_n(\mathbb{F})$ of the $n\times n$ upper triangular matrices in $M_n(\mathbb{F})$ via the correspondence $e_{x_ix_j}\mapsto E_{ij}$, where $E_{ij}$ is an elementary matrix.
\begin{defn}Let $P$ be a locally finite poset and $G$ a group. A $G$-grading on $\A$ is \textit{good} if every element $e_{xy}$ in $\A$ is homogeneous. If $P$ is finite, a good grading on $\A$ is \textit{elementary} if, for each $x_i\preceq x_j$ in $P$,
there are unique elements $g_i,g_j\in G$ satisfying $ e_{x_i x_j}\in \A_{g_i^{-1}g_j}.$
\end{defn}
Next we extend the notion of elementary grading to arbitrary incidence algebras. Let $G^P$ denote the set of all functions from $P$ to $G$ and, for $\theta\in G^P$, denote the image $\theta(x)$ of $x\in P$ simply by $\theta_x$.
\begin{lem}Let $P$ be a locally finite poset, $G$ a group, and $\theta\in G^P$. Then, for each $g\in G$,
\begin{equation}
\A^{\theta}(g) = \{f\in \A\,|\,f(x,y)=0,\text{ whenever $x\preceq y$ and $\theta_x^{-1}\theta_y\neq g$}\}
\end{equation}
is a subspace of $\A$ and the sum $\A^{\theta}$ of all such subspaces is direct. Furthermore,
$\A^{\theta}(g) \A^{\theta}(h)\subseteq \A^{\theta}(gh)$, for all $g,h\in G$.
\end{lem}

\begin{proof}
It is enough to check the third claim. Let $g,h\in G$, $f_1\in \A^{\theta}(g)$, $f_2\in \A^{\theta}(h)$, $x,y\in P$, $x\preceq y$ be arbitrary. For each $z\in [x,y]$ such that
$\theta_x^{-1}\theta_y = (\theta_x^{-1}\theta_z)(\theta_z^{-1}\theta_y)\neq gh$ then $\theta_x^{-1}\theta_z \neq g$ or $\theta_z^{-1}\theta_y \neq h$, so $f_1(x,z)=0$ or $f_2(z,y)=0$. Therefore $f_1f_2\in \A^{\theta}(gh)$.
\end{proof}

\begin{defn}Let $P$ be a locally finite poset and $G$ a group. A $G$-grading on $\A=\oplus_{g\in G}\A_g$ is \textit{elementary} if there exists $\theta\in G^P$ such that $\A_g=\A^{\theta}(g)$ for all $g\in G$ (in which case $\A=\A^{\theta}$).
\end{defn}

Let $\theta\in G^P$ and consider $G_\theta = \{ \theta_x^{-1}\theta_y\,|\,x,y\in P,\,x\preceq y\}\subset G.$ Define
\[
 G_e^P =\{ \theta\in G^{P}\,|\,\text{$G_\theta$ is finite}\}.
\]

\begin{thm}\label{thm:grading:charac}Let $P$ be a locally finite poset, $G$ a group. Then, $\A$ admits an elementary grading for some $\theta\in G^P$ if, and only if, $\theta\in G_e^P$. Moreover, every elementary grading is
good.
\end{thm}
\begin{proof}
First assume that $\theta\in G_e^{P}$, so $G_\theta=\{g_1,\ldots,g_k\}$ is finite. Given $f\in \A$, for each $i=1,2,\ldots,k$, define
\[ f_i(x,y) = 
\begin{cases}
f(x,y), & \text{if $x\preceq y$ and $\theta_x^{-1}\theta_y=g_i$}\\
0, & \text{otherwise.}
\end{cases}
\]
It is clear that $f_i\in \A^{\theta}(g_i)$. Also, given $x\preceq y$ in $P$, $\theta_x^{-1}\theta_y=g_j$, for some index $j$, hence
\[
\sum_{i=1}^k f_i(x,y) = f_j(x,y) = f(x,y).
\]
It follows that $\A\subset \A^{\theta}$, thus $\A = \oplus_{g\in G}\A^{\theta}(g)$ is an elementary $G$-grading. 
%Obviously, $e_{xy}\in \A(\theta_x^{-1}\theta_y)$, for all $x\preceq y$ in $P$, hence this grading is good.

On the other hand, suppose that $\A = \A^{\theta}$, so $\A$ is endowed with an elementary $G$-grading, and $\A_g=\A^{\theta}(g)$ for all $g\in G$.
For any $x\preceq y$ and $u\preceq v$ in $P$, $\theta_x^{-1}\theta_y\neq \theta_u^{-1}\theta_v$ implies $u\neq x$ or $v\neq y$, so $e_{xy}(u,v)=0$.
This shows $e_{xy}\in\A_{\theta_x^{-1}\theta_y}$, thus the grading
is good. 

The algebra $\A$ sure contains the function $\zeta$ given by $\zeta(x,y)=1$ if $x\preceq y$ in $P$, and 0 otherwise.
It has a unique homogeneous decomposition
$\zeta=\zeta_1+\zeta_2+\cdots+\zeta_k$, where the $g_i \in G$ are distinct and the $\zeta_i\in \A_{g_i}$ are nonzero, for all $i$. If $\A_g\neq 0$, then there exist $x\preceq y$ in $P$, such that $g=\theta_x^{-1}\theta_y$. Since $\zeta(x,y)\neq 0$, then $\zeta_i(x,y)\neq 0$ for at least one index $i$, and $\zeta_i(x,y)e_{xy} = e_{xx}\zeta_i e_{yy} \in \A_1 \A_{g_i} \A_1\subset \A_{g_i}$,
so $e_{xy}\in \A_{g}\cap\A_{g_i}$, and $g=g_i$. It follows that $G_\theta\subset \{g_1,\ldots,g_k\}$, thus $\theta \in G_e^P$.
\end{proof}

\begin{rem}\label{decompositionconected} If $P=\cup_{1\leq i\leq k}P_i$ is the decomposition of $P$ into its connected components, it is straightforward to check that
\[
\A^\theta_g=\bigoplus_{1\leq i\leq k} (\A^{\theta_i}_{P_i})_g,
\]
where each $\theta_i = \theta|_{P_i} \in G_e^{P_i}$ defines on $\A_{P_i}=I(P_i,\mathbb{F})$  the elementary grading induced by that on $A^\theta$. For each such $P_i$, the group $G$ acts on the left of $G^{P_i}$ via $(h\theta_i)(x) = h(\theta_i(x))$, for all $h\in G$, $x\in P_i$. This action obviously leaves $G_{\theta_i}$ invariant and $(\A_{P_i})^{\theta_i}_g=(\A_{P_i})^{h\theta_i}_g$ for all $g\in G$, thus it has no effect on the grading.
Further, this action extends naturally to an action of $G^k=G\times\cdots\times G$ on the left of
$G_e^P$ via $(\mathbf{h}\theta)(x)=(h_i\theta_i)(x)$, if $x\in P_i$,
for all $\mathbf{h}=(h_1,\ldots,h_k)\in G^k$ and $\A^{\theta}_g = \A^{\mathbf{h}\theta}_g$, for all $g\in G$.
\end{rem}

\begin{thm}
Let $P$ be a finite poset with $n$ elements and $G$ a finite group. If $k$ is the number of connected components of $P$, then there are $|G|^{n-k}$ distinct elementary gradings on $\A$.
\end{thm}

\begin{proof}
Let $X=G^P=G^n$, $\mathbf{h}\in G^k$, and $\theta\in G^P$. Since $\theta$ and $\mathbf{h}\theta$ define the same elementary grading on $\A$, the number of distinct gradings is equal to the number $|X\backslash G^k|$ of $G^k$-orbits in $X$. It follows from Burnside's Lemma that
\[
|X\backslash G^k| = \frac{1}{|G^k|}\sum_{\mathbf{h}\in G^k}|X^{\mathbf{h}}|=\frac{|G|^n}{|G|^k}=|G|^{n-k},
\] 
where $X^{\mathbf{h}}=\{x\in X\,|\,\mathbf{h}x=x\}$.
\end{proof}

In \cite[Theorem 10]{Jones} the same result was proved using a different argument.

We also want to consider the action of $\Aut (P)$ on the right of $G^P$ given by $(\theta\sigma)(x)=\theta(\sigma^{-1}(x))$, for all $\sigma\in\Aut (P)$, $\theta\in G^P$ and $x\in P$. It is clear this action commutes with the
action of $G^k$ and it leaves $G_\theta$ invariant, hence it induces an
action on the right of $G_e^P$.

\begin{thm}\label{graded isomorphism} If $P$ is a locally finite poset with $k$ connected components, $G$ is a group, $\theta\in G^P_e$, and $\sigma\in \Aut (P)$, then the induced automorphism $\hat\sigma$ of $\A$ is a graded isomorphism from $\A^\theta$ onto $\A^{\theta\sigma}$.
\end{thm}

\begin{proof}
Let $0\neq f\in \A^\theta_g$ for some $g\in G$. Since $e_{xx}fe_{yy}=f(x,y)e_{xy}$, then $e_{xy}\in \A^\theta_g$ whenever $f(x,y)\neq 0$. Similarly,
\begin{align*}
e_{\sigma(x)\sigma(x)}\hat\sigma(f)e_{\sigma(y)\sigma(y)} &= \hat\sigma(f) (\sigma(x),\sigma(y))e_{\sigma(x)\sigma(y)} \\
&= f(\sigma^{-1}\sigma(x),\sigma^{-1}\sigma(y))e_{\sigma(x)\sigma(y)}\\
&= f(x,y)e_{\sigma(x)\sigma(y)}.
\end{align*}
Now,
\[
g=\theta(x)^{-1}\theta(y)=\theta(\sigma^{-1}\sigma(x))^{-1}\theta(\sigma^{-1}\sigma(y))=(\theta\sigma)(\sigma(x))^{-1}(\theta\sigma)(\sigma(y)),
\]
hence $e_{\sigma(x)\sigma(y)}\in \A^{\theta\sigma}_g$ whenever $f(x,y)\neq 0$.
Therefore, $\hat\sigma(f) \in \A^{\theta\sigma}_g$.
\end{proof}

Consider the following relation on $G_e^P$: $\theta\sim \mu \iff \theta=\mathbf{h}\mu\sigma$ for some $\mathbf{h}\in G^k$, $\sigma\in\Aut(p)$.
It is easy to check that this is an equivalence relation.

\begin{thm}\label{isomorphism} Let $P$ be a locally finite poset, $G$ a group, and $\theta,\mu\in G^P_e$. Then $\A^\theta\cong \A^\mu$ as elementary $G$-graded algebras if, and only if, $\theta\sim\mu$.
\end{thm}

\begin{proof}
If $\theta\sim\mu$, then $\mu=\mathbf{h}\theta\sigma$ for some $\mathbf{h}\in G^k$ ($k$ being the number of connected components of $P$) and $\sigma\in\Aut (P)$. Since $\A^\theta = \A^{\mathbf{h}\theta}$, it follows from the previous theorem that $\hat\sigma$ is a graded isomorphism between $\A^\theta$ and $\A^{\mathbf{h}\theta\sigma}=\A^{\mu}$.

On the other hand, suppose that $\varphi:\A^\theta\longrightarrow \A^\mu$ is a
graded isomorphism. From Theorem \ref{thm:decomp:automorphism}, we know that
$\varphi = \psi_r M_s \hat\sigma$ for some $r,s\in \A$ and unique $\sigma\in \Aut (P)$. Since $s\in \A$ is a multiplicative element, then $M_s$ is a graded isomorphism. Indeed, if $f\in \A^\theta_g$ for some $g\in G$, $x\preceq y$ in $P$, and $\theta_x^{-1}\theta_y\neq g$, then it follows from Theorem \ref{thm:grading:charac} that $M_s(f)(x,y)=s(x,y)f(x,y)=0$, hence $M_s(f)\in \A^\theta_g$. So far, we have that
\[        \xymatrix@-1.75pc{\A^\theta \ar[rr]^{\hat\sigma} \ar[dd]^{\varphi} & &
                \A^{\theta\sigma}\ar[dd]^{M_s}\\
                &  & \\
                \A^\mu  &  & \A^{\theta\sigma} \ar[ll]_{\psi_r}}
\]
is a commutative diagram where $\varphi$, $\hat\sigma$ and $M_s$ are graded isomorphisms, hence so is $\psi_r=\varphi \hat\sigma^{-1} M_s^{-1}$. Notice that $\hat\sigma^{-1}=\widehat{\sigma^{-1}}$ and $M_s^{-1}$ is multiplicative (cf. remark \ref{rem:auto:mult}). Since $r$ is invertible, 
\[
r(x,x)r^{-1}(y,y) e_{xy} = \psi_r(e_{xy})(x,y)e_{xy} = e_{xx}\psi_r(e_{xy})e_{yy},
\]
hence, for $e_{xy}\in \A^{\theta\sigma}_g$, then $\psi_r(e_{xy})\in \A^{\sigma\theta}_g \cap \A^\mu_g$, so $g=(\theta\sigma)_x^{-1}(\theta\sigma)_y=\mu_x^{-1}\mu_y$.

To see that $\theta\sim\mu$, first assume that $P$ has a single connected component, choose some $x_0\in P$ and define $h=\mu_{x_0}(\theta\sigma)^{-1}_{x_0}$. For $y\in P$ arbitrary, there is a sequence $x_0,x_1,\ldots,x_{m+1}=y$ in $P$ such that either $x_i\preceq x_{i+1}$ or $x_{i+1}\preceq x_i$, for all $i=0,1,\ldots,m$. In either case, we have $(\theta\sigma)_{x_i}^{-1}(\theta\sigma)_{x_{i+1}}=\mu_{x_i}^{-1}\mu_{x_{i+1}}$. It follows that $\mu_{x_0}=h(\theta\sigma)_{x_0}$, $\mu_{x_1}=\mu_{x_0}(\theta\sigma)_{x_0}^{-1}(\theta\sigma)_{x_1}=h(\theta\sigma)_{x_1}$, and so on, thus $\mu_y=h(\theta\sigma)_y$. Since $y$ is arbitrary, we conclude that $\mu= h\theta\sigma$. Proceeding in the same way for each connected component, we obtain $\mathbf{h}\in G^k$ such that $\mu=\mathbf{h}\theta\sigma$. This completes the proof.
\end{proof}

When $P$ and $G$ are finite, we can count the number of inequivalent elementary gradings on $\A$.

\begin{cor}If $P$ is a finite poset with $k$ connected components and $G$ is a finite group, then the number of inequivalent elementary gradings on $\A$ is
\[
n_G = \frac{1}{|G|^k|\Aut (P)|}\sum_{(\mathbf{h},\sigma)\in G\times \Aut (P)}|X^{(\mathbf{h},\sigma)}|
\]
where $X^{(\mathbf{h},\sigma)}=\{\theta\in G_e^P\,|\,\mathbf{h}\theta\sigma\sim \theta\}$.
\end{cor}

For instance, for $P=C_n$ is a finite chain, then $\A\cong UT_n$, $\Aut (P)$ is trivial, $k=1$, and the number of inequivalent elementary gradings is $|G|^{n-1}$. 

\section{Graded Polynomial Identities on Incidence Algebras}
\qquad Now, we shall study the graded polynomial Identities on Incidence algebras. In \cite[Theorem 8.2.5]{Spiegel} it was proved that $\IP$ is a PI-algebra if and only if $P$ is bounded. Throughout this section we shall assume that the poset $P$ is locally finite and bounded. We shall also assume that $\mathbb{F}$ is an infinite field of characteristic zero and $G$ is a group.

Let $X=\cup_{g\in G} X_g$ be the union of the disjoint countable sets $X_g=\{x_1^g,x_2^g,\dots\}$. The free associative algebra $\FX$ freely generated over $\mathbb{F}$ by $X$ is equipped in a natural way with a structure of $G$-graded algebra, namely that in which $\deg(x_i^g)=g$ for every $x_i ^g\in X_g$, and $\deg(x_{i_1}^{g_1}x_{i_2}^{g_2}\dots,x_{i_t}^{g_t})=g_1g_2\dots g_t$. So $\FX$ is the free $G$-graded algebra freely generated by $X$.
Let $\Phi(x_{i_1}^{g_1},x_{i_2}^{g_2},\dots,x_{i_t}^{g_t})\in \FX$ be a polynomial. If $R=\oplus R_g$ is a $G$-graded algebra then $\Phi$ is a $G$-graded polynomial identity for $R$ if $\Phi(r_{i_1}^{g_1},r_{i_2}^{g_2},\dots,r_{i_t}^{g_t})=0$ in $R$ for all homogeneous elements $r_{i_k}^{g_s}\in R_{g_s}$. The ideal $T_G(R)$ in $\FX$ of all $G$-graded polynomial identities of $R$ is closed under all $G$-graded endomorphisms of $\FX$. Such ideals are called $G$-graded $T$-ideals.

In \cite[Theorem 7]{Berele}, the following was proved for the ordinary polynomial identities:
\begin{thm}
Let $\A=\IP$, then
\[
T(\A)=\bigcap \{T(\A_C)\,|\,C \subseteq P \text{ is a maximal chain }\}.
\]
where $\A_C=I(C,\mathbb{F})$.
\end{thm}
Motivated by this result, we prove its graded version:
\begin{thm}\label{t-ideal}
Let $G$ a group and $\theta\in G_e^P$ an elementary grading on $\A$. Then 
\[
T_G(\A)^{\theta}=\bigcap T_G(\A_{C_i})^{\theta_i}
\]
for all maximal chains $C_i\subset P$.
\end{thm}
\begin{proof}
If $\Phi$ is a graded polynomial identity for $\A^{\theta}$ then, in particular, $\Phi$ is a polynomial identity for $\A_{C_i}^{\theta_i}$, where $\theta_i$ is the elementary grading induced by $\theta$ on $\A_{C_i}$. This proves the inclusion ($\subseteq$). To prove the other inclusion, let 
\[
J=\bigcap T_G(\A_{C_i})^{\theta_i}.
\]
Since \emph{char} $\mathbb{F}=0$, we have that $J$ is generated by the set of its multilinear elements, and we consider 
\[
\Phi=\Phi(f_1,f_2,\dots, f_m) \in J,
\]
a graded multilinear polynomial. Suppose that $\Phi \notin T_G(\A)^{\theta}$. Because of the multilinearity, there must exist elements $e_{u_1v_1},\penalty 0 e_{u_2v_2},\dots,e_{u_mv_m} \penalty 0 \in \A^{\theta}$ such that $\Phi(e_{u_1v_1},e_{u_2v_2},\dots,e_{u_mv_m})\neq 0$.
Since $\Phi$ is multilinear and at least one of its monomials does not vanish, then\\ $\{u_1,v_1,u_2,v_2,\dots u_m,v_m\}$ forms a chain in $P$ which, in turn, is contained in some maximal chain $C_j$. Since $\Phi \in J$, in particular $\Phi \in T_G(\A_{C_j})^{\theta_j}$, therefore $\Phi(e_{u_1v_1},\dots,e_{u_mv_m})= 0$, which is a contradiction.
\end{proof}

The next example shows that it is possible to have $T_G(\A)^{\theta}=T_G(\A)^{\mu}$ even though $\A^{\theta}\ncong\A^{\mu}$ as graded algebras.
\begin{ex}
Let $P=\{p_1,p_2,p_3,p_4\}$ be a poset with $p_1\preceq p_4,\ p_2\preceq p_3,\ p_2\preceq p_4$.
\begin{figure}[h]
\begin{center}\begin{tabular}{cccc}
\begin{tikzpicture}[level/.style={sibling distance = 5cm/#1,
  level distance = 1.5cm}] 
\node [arn_n] (a) [label=left:{\scriptsize $p_4$}] {};
\node [arn_n] (b) [below of=a,label=left:{\scriptsize $p_1$}] {};
\node [arn_n] (c) [right of=a,label=right:{\scriptsize $p_3$}] {};
\node [arn_n] (d) [below of=c,label=right:{\scriptsize $p_2$}] {};
\draw []  (a) edge (b);
\draw []  (a) edge (d);
%\draw []  (b) edge (d);
\draw []  (c) edge (d);
\end{tikzpicture}
&
\begin{tikzpicture}[level/.style={sibling distance = 5cm/#1,
  level distance = 1.5cm}] 
\node [arn_n] (a) [label=left:{\scriptsize $p_4$}] {};
\node [arn_n] (b) [below of=a,label=left:{\scriptsize $p_1$}] {};
\draw []  (a) edge (b);
\end{tikzpicture}
&
\begin{tikzpicture}[level/.style={sibling distance = 5cm/#1,
  level distance = 1.5cm}] 
\node [arn_n] (a) [label=left:{\scriptsize $p_4$}] {};
\node [arn_n] (b) [below of=a,label=left:{\scriptsize $p_2$}] {};
\draw []  (a) edge (b);
\end{tikzpicture}
&
\begin{tikzpicture}[level/.style={sibling distance = 5cm/#1,
  level distance = 1.5cm}] 
\node [arn_n] (a) [label=left:{\scriptsize $p_3$}] {};
\node [arn_n] (b) [below of=a,label=left:{\scriptsize $p_2$}] {};
\draw []  (a) edge (b);
\end{tikzpicture}
\end{tabular}\end{center}
\caption{Poset $P$ and its maximal chains $C_1, C_2$ and $C_3$.}
\end{figure}
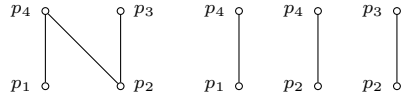%
Notice that $C_1=\{p_1,p_4\},C_2=\{p_2,p_4\},C_3=\{p_2,p_3\}$ are the maximal chains in $P$. Consider $G=\{1,h,h^2\}$ and the elementary gradings given by $\theta =(1,h,h^2,1)$ and $\mu =(1,h^2,h,1)$. Since $\Aut (P)$ is trivial, $\theta \nsim \mu$, so $\A^{\theta}$ and $\A^{\mu}$ are not isomorphic as graded algebras. However, their graded polynomial identities coincide. Indeed, the induced grading on $C_1$ is $\theta|_{C_1}=(1,1)$, which is trivial, and it follows that $\A_{C_1}^{\theta|_{C_1}}=\langle e_{11},e_{14},e_{44}\rangle$, so its graded identities are the same as those of the algebra $UT_2(\mathbb{F})$. Moreover, since all maximal chains have length $2$ and $T_G(UT_2(\mathbb{F}))\subseteq T_G(\A_{C_i})^{\theta|_{C_i}}$ for all $i=1,2,3$, we have
\[
T_G(\A_{P})^{\theta}=\bigcap_{i=1,2,3}T_G(\A_{C_i})^{\theta_i}=T_G(UT_2(\mathbb{F})).
\]
The same holds for the grading $\mu$, since $\mu|_{C_1}=(1,1)$. So,
\[
T_G(\A)^{\theta}=T_G(\A)^{\mu}.
\]
\end{ex}

\begin{thm}
Let $G$ be a group and $\theta,\mu \in G_e^P$ endowing elementary gradings on $\A$. Suppose that $\Aut(P)$ acts transitively on the maximal chains of $P$ and $T_G(\A)^{\theta}=T_G(\A)^{\mu}$. Then $\A^\theta$ and $\A^\mu$ are isomorphic as graded algebras.
\end{thm}
\begin{proof}
First suppose that $P$ has a single connected component. If $\Phi \in T_G(\A)^{\theta}$ ($=T_G(\A)^{\mu}$) then, by Theorem \ref{t-ideal}, we have that $\Phi$ is a polynomial identity for both $\A_{C_i}^{\theta_i}$ and $\A_{C_j}^{\mu_j}$. Since $\Aut(P)$ acts transitively on the maximal chains and $P$ is bounded, each maximal chain has the same length, say $n$, and there exists $\sigma_{ij} \in \Aut(P)$ such that $\sigma_{ij}(C_i)=C_j$. By Theorem \ref{graded isomorphism} we have that $\Phi$ is still a polynomial identity for $\A_{C_j}^{\theta_i\sigma_{ij}}$. Moreover, restricted to a maximal chain $C_j\subset P$, we have that $\A_{C_j}^{\mu_j}\simeq UT_n(\mathbb{F})$. In particular, $\A_{C_j}^{\theta_i\sigma_{ij}}$ and $\A_{C_j}^{\mu_j}$ have the same graded identities. Now, by the second statement of \cite[Theorem 2.3]{Plamen}, $\theta_i\sigma_{ij}\sim\mu_j$. Since $i,j$ are arbitrary, we can choose $\sigma\in \Aut(P)$ such that
\[
\sigma(x)=\begin{cases}
\sigma_{ij}(x), &\mbox{if \ensuremath{x\in C_i}} \\
x, & \mbox{if \ensuremath{x\notin C_i}}. 
\end{cases}
\]
So, by remark \ref{decompositionconected}, $\theta\sigma\sim\mu$ by construction and there exist $h\in G$ such that $\mu=h\theta\sigma$. Proceeding in the same way for each connected component of $P$, we obtain $\mathbf{h}\in G^k$ such that $\mu=\mathbf{h}\theta\sigma$. From Theorem \ref{isomorphism}, it follows that $\A^\theta\cong\A^\mu$.
\end{proof}

%====================================================================
%====================================================================

\end{document}